\documentclass[10pt]{article}
\font\smallit=cmti10

\usepackage{amssymb,latexsym,amsmath,epsfig,amsthm} 
\makeatletter

\renewcommand\section{\@startsection {section}{1}{\z@}
{-30pt \@plus -1ex \@minus -.2ex}
{2.3ex \@plus.2ex}
{\normalfont\normalsize\bfseries}}

\renewcommand\subsection{\@startsection{subsection}{2}{\z@}
{-3.25ex\@plus -1ex \@minus -.2ex}
{1.5ex \@plus .2ex}
{\normalfont\normalsize\bfseries}}

\renewcommand{\@seccntformat}[1]{\csname the#1\endcsname. }

\makeatother

\newtheorem{theorem}{Theorem}[section]
\newtheorem*{theorem*}{Theorem}
\newtheorem{proposition}[theorem]{Proposition}
\newtheorem{lemma}[theorem]{Lemma}
\newtheorem{corollary}[theorem]{Corollary}
\newtheorem{definition}[theorem]{Definition}
\theoremstyle{definition}
\newtheorem{remark}[theorem]{Remark}
\newtheorem{example}[theorem]{Example}

\begin{document}

%%%%%%%%%%%%%%%%%%%%%%%%%%%%%%%
\newcommand\+[1]{\mathcal{#1}}
\newcommand{\N}{{\mathbb N}}     % Non negative integers

\newcommand{\Z}{{\mathbb Z}}     % Integers
\newcommand{\Q}{{\mathbb Q}}     % Rationals
\newcommand{\RR}{{\mathbb R}}    % Reals
\newcommand{\C}{{\mathbb C}}     % Complexe
%%%%%%%%%%%%%%%%%%%%%%%%%%%%%%%
\newcommand{\dom}{\textit{Domain}}     
\newcommand{\suc}{\textit{Suc}}     
\newcommand{\expp}{\textit{exp3}}     
\newcommand{\card}{\textit{card}}     
\newcommand{\lcm}{\textit{lcm}}     
\newcommand{\round}{\textit{round}}     
\newcommand{\change}[1]{\marginpar{\color{red} {\bf{\Large  Z \ }{#1}}}}

\begin{center}
{\bf\uppercase{Characterizing congruence preserving functions}
$\Z/n\Z\to\Z/m\Z$ \uppercase{via rational polynomials}}
%----------------
\vskip 20pt
{\bf Patrick C\'EGIELSKI%
\newcounter{thanks}
\setcounter{thanks}{\value{footnote}}
\footnote{Partially supported by TARMAC ANR agreement  12 BS02 007 01.}}\\
{\smallit LACL, EA 4219, Universit\'e Paris-Est Cr\'eteil, France}\\
{\tt cegielski@u-pec.fr}\\ 
%----------------
\vskip 10pt
{\bf Serge GRIGORIEFF
\setcounter{thanks}{\value{footnote}}%
\footnotemark[\value{thanks}]}\\
{\smallit LIAFA, CNRS and Universit\'e Paris-Diderot, France}\\
{\tt seg@liafa.univ-paris-diderot.fr}\\ 
%----------------
\vskip 10pt
{\bf Ir\`ene  GUESSARIAN
\setcounter{thanks}{\value{footnote}}%
\footnotemark[\value{thanks}]
\footnote{Emeritus at UPMC Universit\'e Paris 6. Corresponding author}}\\
{\smallit LIAFA, CNRS and Universit\'e Paris-Diderot, France}\\
{\tt ig@liafa.univ-paris-diderot.fr}\\ 
\end{center}
\vskip 30pt

%\centerline{\smallit Received: , Revised: , Accepted: , Published: } % We will fill in the dates
\vskip 30pt

\centerline{\bf Abstract}

\noindent
We introduce a basis of rational polynomial-like functions
$P_0,\ldots,P_{n-1}$
for the free module of functions $\Z/n\Z\to\Z/m\Z$.
We then characterize the subfamily
of  congruence preserving functions
as the set of linear combinations of the functions $\lcm(k)\,P_k$
where $\lcm(k)$ is the least common multiple of $2,\ldots,k$
(viewed in $\Z/m\Z$).
As a consequence, when $n\geq m$, the number of such functions is independent of $n$.

%\pagestyle{myheadings}
%\markright{\smalltt INTEGERS: 15 (2015)\hfill}
%\thispagestyle{empty}
%\baselineskip=12.875pt
%\vskip 30pt 

%%%%%%%%%%%%%%%%%%%%%%%%%%%%
%%%%%%%%%%%%%%%%%%%%%%%%%%%%
%%%%%%%%%%%%%%%%%%%%%%%%%%%%
\section{Introduction}
%%%%%%%%%%%%%%%%%%%%%%%%%%%%
%%%%%%%%%%%%%%%%%%%%%%%%%%%%
%%%%%%%%%%%%%%%%%%%%%%%%%%%%
%
%
%%%%%%%%%%%%%%%%%%%%%%%%%%%%

%%%%%%%%%%%%%%%%%%%%%%%%%%%%
%
%
The notion of congruence preserving function $\Z/n\Z\to\Z/m\Z$
was introduced in Chen~\cite{ch1995} and studied in 
Bhargava~\cite{Bh1997}.
\begin{definition}\label{def:Bhargava}
Let $m,n\geq1$.
A function $f:\Z/n\Z\to\Z/m\Z$ is said to be {\em congruence 
preserving } if  for all $d$ dividing $m$
\begin{multline}\label{eq:Bha}
 \forall a, b\in\{0,\ldots,n-1\}\qquad
a\equiv b \pmod d\Longrightarrow
f(a) \equiv f(b)\pmod d\quad
\end{multline}
\end{definition}
\begin{remark}\label{rk:Bhargava}
1. If $n\in\{1,2\}$ or $m=1$ then every function
$\Z/n\Z\to\Z/m\Z$ is trivially congruence preserving.
\\
2. Observe that  since $d$ is assumed to divide $m$,
equivalence modulo $d$ is a congruence on $(\Z/m\Z,+,\times)$.
However, since $d$ is not supposed to divide $n$,
equivalence modulo $d$ may not be a congruence on $(\Z/n\Z,+,\times)$.
\end{remark}
\begin{example}\label{ex:Bhargava}
1. For functions $\Z/6\Z\to\Z/3\Z$, condition~\eqref{eq:Bha}
reduces to the conditions
$f(3)\equiv f(0)\pmod 3$, $f(4)\equiv f(1)\pmod 3$, $f(5)\equiv f(2)\pmod 3$.
\\
2. For functions $\Z/6\Z\to\Z/8\Z$, condition~\eqref{eq:Bha}
reduces to 
$f(2)\equiv f(0)\pmod 2$,
$f(3)\equiv f(1)\pmod 2$,
$f(4)\equiv f(0)\pmod 4$,
$f(5)\equiv f(1)\pmod 4$.
\end{example}

A formal polynomial $F(X)\in(\Z/m\Z)[X]$
has no canonical interpretation as a function $\Z/n\Z\to\Z/m\Z$
when $m$ does not divide $n$~:
indeed $a\equiv b\pmod n$ does not imply  $F(a)\equiv F(b) \pmod m$.

According to Chen \cite{ch1995,ch1996}  a function $f:\Z/n\Z\to\Z/m\Z$ is said 
to be {\em polynomial} if there exists some polynomial $F\in\Z[X]$
such that,
for all $a\in\{0,\ldots,n-1\}$, $f(a)\equiv F(a)\pmod m$.
Chen also shows that there can be congruence preserving functions which are not polynomial.
Using counting arguments,
Bhargava \cite{Bh1997} characterizes the ordered pairs $(n,m)$
such that every congruence preserving function $f:\Z/n\Z\to\Z/m\Z$
is polynomial.

\medskip

In Section \ref{s:poly rat} we introduce a notion of
rational polynomial function $f:\Z/n\Z\to\Z/m\Z$
based on polynomials with rational coefficients
which map integers to integers.
We observe that the free $\Z/m\Z$-module of functions
$f\colon\Z/n\Z\to\Z/m\Z$ admits a 
basis of such rational polynomials $P_0,\ldots,P_{n-1}$
where $P_k$ has degree $k$.
Indeed, every function $\Z/n\Z\to\Z/m\Z$ is rational polynomial
of degree at most $n-1$.

In Section \ref{s:main} we prove the main theorem of this paper:
 congruence preserving functions
$f\colon\Z/n\Z\to\Z/m\Z$
are the $\Z/m\Z$-linear combinations of the functions $\lcm(k)P_k$
where $\lcm(k)$ is the least common multiple of $2,\ldots,k$
(viewed in $\Z/m\Z$).
The proof adapts the techniques of our paper \cite{cgg14},
exploiting similarities between Definition~\ref{def:Bhargava}
and the condition studied in \cite{cgg14}
for functions $f:\N\to\Z$
(namely, $x \! - \! y$ divides $f(x)\! -\! f(y)$ for all $x,y\in\N$).

In Section \ref{s:count} we get a by-product of our characterization:
every congruence preserving function $f\colon\Z/n\Z\to\Z/m\Z$
is rational polynomial for a polynomial of degree less than 
the minimum between $n$ and $\mu(m)$
(the largest prime power dividing $m$).
We also use our main theorem to count the congruence preserving functions $\Z/n\Z\to\Z/m\Z$.
We thus get an expression equivalent to that
obtained by Bhargava in \cite{Bh1997}
and which makes apparent the fact that, for $n\geq\mu(m)$
(hence for $n\geq m$),
this number depends only on $m$ and is independent of $n$. 
%

%%%%%%%%%%%%%
%%%%%%%%%%%%%
\section{Representing functions $\Z/n\Z\to\Z/m\Z$
by rational polynomials}\label{s:poly rat}
%%%%%%%%%%%%%
%%%%%%%%%%%%%
Some polynomials in $\Q[X]$
(i.e. polynomials with rational coefficients)
happen to map  $\N$ into $\N$, i.e. they take integer values  for all arguments in $\N$. 
\begin{definition}\label{def:Pk}
For $k\in\N$, let $P_k\in\Q[X]$ be the following polynomial:
$$
P_k(x)=\dbinom{x}{k}
= \dfrac{\prod_{i=0}^{k-1} (x - i)}{k!}\,.
$$
\end{definition}
\noindent
The $P_k$  are also called binomial polynomials.
We will use in later examples 
\\
$P_0(x)=1$\ \ ,\ \ $P_1(x)=x$\,,\;$P_2(x)=x(x-1)/2$\ \ ,\ \ $P_3(x)={x(x-1)(x-2)}/{6}$\ \ ,\ \ $P_4(x)={x(x-1)(x-2)(x-3)}/{24}$\ \ ,\ \ $P_5(x)={x(x-1)(x-2)(x-3)(x-4)}/{120}$.

\smallskip
\noindent
In 1915, P\'olya \cite{polya1915} used the $P_k$ to give  the following very elegant and elementary characterization of  polynomials which take integer values on the integers.
\begin{theorem}[P\'olya]
A polynomial is integer-valued on $\Z$ iff it can be written as a $\Z$-linear combination of the polynomials 
$P_k$.
\end{theorem}
\smallskip
\noindent
It turns out that the representation of functions $\N\to\Z$ as $\Z$-linear combinations of the $P_k$'s used in \cite{cgg14}
also fits in the case of functions $\Z/n\Z\to \Z/m\Z$~:
every such function is a $(\Z/m\Z)$-linear combination of the $P_k$'s.

\begin{definition}\label{def:rat-pol}
A function $f:\Z/n\Z\to\Z/m\Z$ is {\em rat-polynomial }
if there exists a polynomial with rational coefficients $R\in\Q[X]$
such that
$
\forall a\in\{0,\ldots,n-1\},\ 
f(a)\equiv R(a) \pmod m.
$
The degree of $f$ is the smallest among the degrees of
such polynomials $R$. 

We denote by $P_k^{n,m}$
the rat-polynomial function $\Z/n\Z\to\Z/m\Z$
associated with the polynomial $P_k$ of Definition~\ref{def:Pk}.
When there is no ambiguity, i.e. when $n,m$ are fixed,
$P_k^{n,m}$ will be denoted simply as $P_k$.
\end{definition}
\begin{remark}
In Definition~\ref{def:rat-pol}, the polynomial $R$
 {\em depends }  on the choice
of  representatives of elements of $\Z/n\Z$: e.g.  for $n=m=6$,
$0\equiv 6\pmod{6}$ but $0=P_2(0)\not\equiv P_2(6)=3\pmod{6}$.
The chosen representatives for elements of $\Z/n\Z$ will always be  $\{0,\ldots,n-1\}$.
\end{remark}
We now prove the representation result by the $P_k$'s.
\begin{theorem}\label{uniqueRepresentation}
For any function $f\colon \Z/n\Z \to \Z/m\Z$, there exists a unique sequence $a_0, a_1, \ldots, a_{n-1}$ of elements in $\Z/m\Z$
such that
\begin{equation}\label{eq:main}
f = \sum_{k=0}^{n-1} \; a_k\,P_{k}\quad {\text  with }\quad P_{k}(x) = \displaystyle\frac{\prod_{i=0}^{k-1} (x - i)}{k!}={{x}\choose{k}}
\end{equation} 
\end{theorem}

\begin{proof}
Let us  begin by unicity.
We have $f(0) = a_0$ hence $a_0 = f(0)$.
We have $f(1) = a_0 + a_1$, hence $a_1 = f(1) - f(0)$.
By induction, and noting that $P_{k}(k)=1$, we have $f(k) = Q(k) + a_k.P_{k}(k)=Q(k) + a_k$,
hence we are able determine $a_k$. 

For existence, argue backwards to see that
 this sequence suits.
\end{proof}
\begin{remark}
The evaluation order of   $a_k\,P_k(x)$ in $\Z/m\Z$ is defined as follows:
for $x$ an element of $\Z/n\Z$, we consider
it as an element of $\{0,\ldots,n-1\}\subseteq\N$ and we evaluate
$P_{k}(x) ={ \displaystyle\frac{1}{k!} }\prod_{i=0}^{k-1} (x - i)$
as an element of $\N$, then we consider the remainder modulo $m$, and finally we multiply 
the result by $a_k$ in $\Z/m\Z$. For instance, for $n=m=8$, $4\,P_{2}(3) = 4\displaystyle.\frac{3\times 2}{2} = 4.3 = 4$,
but we might be tempted to evaluate it as
$4\,P_{2}(3) = \displaystyle\frac{4 \times 3 \times 2}{2} 
= \frac{0}{2} = 0$,
which does {\bf not } correspond to our definition. However,  dividing $a_k$ by a factor of the denominator is allowed.
%
%3)  Note that the $P_{k}$ are polynomials with {\em rational} coefficients (possibly but not necessarily some coefficients may belong to $\Z/m\Z$).
\end{remark}
\begin{corollary} \label{basis}(1) Every function $f:\Z/n\Z\to\Z/m\Z$
is rat-polynomial with degree less than $n$.
\\
(2) The family of rat-polynomial functions $(P_k)_{k=0,\ldots,n-1}$
is a basis of the $(\Z/m\Z)$-module of functions
$\Z/n\Z\to\Z/m\Z$.
\end{corollary}

\begin{example}\label{exZ6}
The function $f\colon  \Z/6\Z \to \Z/6\Z$ defined by

\begin{tabular}{cccccc}
$0 \mapsto 0$& $1\mapsto 3$&  $ 2\mapsto 4 $&  $3\mapsto 3$& $ 4\mapsto 0$&  $ 5\mapsto 1$\end{tabular}

\noindent is represented by the rational polynomial   $P_f(x)=3x+4\frac{x(x-1)}{2}$  which can be simplified into
$P_f(x)=3x-x(x-1)$ on $\Z/6\Z$.
\end{example}

\begin{example}  The function $f\colon \Z / 6\Z  = \Z / 8\Z$ 
given by Chen \cite{ch1995} as a non polynomial
congruence preserving function, namely the function defined by
$ f(0) = 0\  , \  
f(1 )=  3 \  , \   f(2 )=  4\  , \  
f(3 )=  1\  , \  
f(4 )=  4\  , \  
f(5 )=  7$,
is represented by the rational polynomial
with coefficients 
$a_0=0\,,\, a_1=3\, ,\, a_2=6\, ,\, a_3=2\, ,\, a_4=4 \, ,\, a_5=4$,  i.e.

\begin{eqnarray*}
f(x)&=& 3 x +6 \frac{x(x-1)}{2} + 2 \frac{x(x-1)(x-2)}{2 }+  4\ \frac{x(x-1)(x-2)(x-3)}{8} \\
&&\qquad+  4\ \frac{x(x-1)(x-2)(x-3)(x-4)}{8}\\
&=& 3 x +3 {x(x-1)} + {x(x-1)(x-2)}+ \ \frac{x(x-1)(x-2)(x-3)}{2} \\
&&\qquad+ \  \frac{x(x-1)(x-2)(x-3)(x-4)}{2}
\end{eqnarray*}
\end{example}  

%%%%%%%%%%%%%%%%
\section{Characterizing congruence preserving functions
$\Z/n\Z \to \Z/m\Z$}\label{s:main}
%%%%%%%%%%%%%%%%%
%
Congruence preserving  functions
$f\colon\Z/n\Z\to \Z/m\Z$ can be  characterized by a simple condition on the coefficients
of the rat-polynomial representation of $f$ given in
Theorem \ref{uniqueRepresentation}.

%---------------------------------------
\subsection{Main theorem}\label{ss:main}
%---------------------------------------
%
%
As in \cite{cgg14} we need  the unary least common  multiple.
\begin{definition}\label{def:lcm}
For $k\in\N\setminus\{0\}$, $\lcm(k)$ is the least common multiple of
all positive integers less than or equal to $k$.
By convention, $\lcm(0)=1$.
\end{definition}
\begin{theorem}\label{thm:main}
Consider a function $f\colon \Z/n\Z \to \Z/m\Z$, $n,m\geq 1$, and let
\ $f = \sum_{k=0}^{n-1} a_k\,P_{k}$ be its representation
given by equation \eqref{eq:main} of Theorem~\ref{uniqueRepresentation}.
Then $f$ is congruence preserving 
if and only if $\lcm(k)$, considered  in $\Z/m\Z$,
divides $a_k$ for all $k=0,\ldots,n-1$\,.
\end{theorem}

For proving  Theorem \ref{thm:main} we will need some relations involving binomial coefficients
and the unary $\lcm$ function; these relations are stated in
 the next three lemmata  (proven in  \cite{cgg14}).
\begin{lemma}\label{lemme.p}
If \ $0\leq n-k<p\leq n<m$ then $p$ divides $\lcm(k)\dbinom{n}{k}$.
\end{lemma}

\begin{lemma}\label{l:n divise Ankb}
If $n,k,b\in\{0,1,\ldots,m-1\}$ and $k\leq b$ then $n$ divides
$A^n_{k,b}=\lcm(k) \left(\dbinom{b+n}{k}-\dbinom{b}{k}\right)$.
\end{lemma}
The following is an immediate consequence of Lemma \ref{l:n divise Ankb}
(set $a=b+n$).
\begin{lemma}\label{lemme.a-b}
If $m>a\geq b $ then $a-b$ divides
$\lcm(k) \left(\dbinom{a}{k}-\dbinom{b}{k}\right)$
for all $k\leq b$.
\end{lemma}

Besides these  lemmata, we shall use a classical result in $\Z/m\Z$.

\begin{lemma}\label{l:lcm modular}
Let $a_1,\ldots,a_k\geq1$ and $c$ be their least common multiple.
If $a_1,\ldots,a_k$ all divide $x$ in  $\Z/m\Z$
then so does $c$.
\end{lemma}
\begin{proof}
It suffices to consider the case $k=2$ since the passage to any $k$
is done via a straightforward induction.
Let $c=a_1b_1=a_2b_2$ with $b_1,b_2$ coprime.
Let $t,u$ be such that $x=a_1t=a_2u$ in $\Z/m\Z$\,.
Then $x\equiv a_1t\equiv a_2u\pmod m$.
Using B\'ezout identity, let $\alpha,\beta\in\Z$ be such that
$\alpha b_1+\beta b_2=1$.
Then 
$$
c(t\alpha+u\beta) 
\ =\  a_1b_1t\alpha + a_2b_2u\beta\\
\ \stackrel{\bmod k}{\equiv}\  x\alpha b_1+x\beta b_2
\ =\  x
$$
hence $c(t\alpha+u\beta) = x$,
proving that $c$ divides $x$ in $\Z/m\Z$.
\end{proof}
%
%We can now state and prove our main theorem.

 \begin{proof}[Proof of Theorem \ref{thm:main}]
  {\bf ``Only if" part.}%----------------------------------------------------------------
\  Assume $f:\Z/n\Z\to\Z/m\Z$ is congruence preserving and consider its decomposition
$f(x)=\sum_{k=0}^{n-1} a_k .P(x)$ given by Theorem~\ref{uniqueRepresentation}.
We show that
$\lcm(k)$ divides $a_k$ for all $k<n$. %This is trivial for $k$ such that $gcd(a,m)=1$, so the proof concerns all the other $k$'s.
\smallskip\\
{\it{\normalfont\bf Claim 1.} For all $m>k\geq1$, $k$ divides $a_k$.}

\begin{proof} By induction on $k$.
Recall that $f(k)= a_i \sum_{i= 0}^{n-1} \binom{k}{i} =a_i \sum_{i= 0}^k \binom{k}{i} $ by noting that $\binom{k}{i} =0$ for $i>k$.\\
{\it Induction Basis:}
The case $k=1$ is trivial.
For $k=2$, if 2 does not divide $m$ then 2 is invertible in $\Z/m\Z$, hence  $2$ divides $a_2$.
Otherwise, observe that, as  $2$ divides $2-0$, and $f$ is congruence preserving,  $2$ divides $f(2)-f(0)= 2a_1 +a_2$ hence $2$ divides $a_2$.
\\
{\it Induction:} assuming that $\ell$ divides $a_\ell$ for every $\ell\leq k$,
we prove that $k+1$ divides $a_{k+1}$.  Assume first that $k+1$ divides $m$, then
\begin{eqnarray}\label{eq:k+1-0}
f(k+1) -f(0)&=&
(k+1) a_1 + \left(\sum_{i=2}^k \binom{k+1}{i} a_i\right)+a_{k+1}\notag
\\
&=&(k+1) a_1 + 
\left(\sum_{i=2}^k (k+1)\;\frac{a_i}{i}\;\binom{k}{i-1}\right)
+a_{k+1}
\end{eqnarray}
By the induction hypothesis, ${a_i\over i}$ is an integer for $i\leq k$.
Since $f$ is congruence preserving, $k+1$ divides $f(k+1) -f(0)$
hence $k+1$ divides the last term $a_{k+1}$ of the sum. 

Assume now that $k+1$ does not divide $m$, then
$k+1=a\times b$ with $b$ dividing $m$ and $a$ coprime with $m$.  Hence $a$ is invertible in $\Z/m\Z$ and,  by the 
congruence preservation property of $f$, $b$ divides $f(k+1) -f(0)$ ; as  $b$ divides  $k+1$, equation \eqref{eq:k+1-0} implies that  $b$ divides $a_{k+1}$, and $a\times b$ also divides $a_{k+1}$ (by invertibility of $a$  and Lemma \ref{l:lcm modular}).
\end{proof}
%\hfill
%$\square$
%
\smallskip
\noindent
{\it{\normalfont\bf Claim 2.} For all $1\leq p\leq k$, $p$ divides $a_k$.
Thus, $\lcm(k)$ divides $a_k$.}
\begin{proof}
The last assertion of  Claim 2 is a direct application of
Lemma~\ref{l:lcm modular} to the first assertion
which we now prove.
The case $p=1$ is trivial.
We prove the  $p\geq2$ by induction on $p$.
\\
\textbullet\;{\it Basic case $p=2$ : $2$ divides $a_k$ for all $k\geq2$.}  If 2 does not divides $m$, then 2 is invertible and divides all numbers in $\Z/m\Z$;  assume that 2 divides $m$. 
We argue by induction on $k\geq2$.
\\
- {\it Basic case (of the basic case).} 
Apply Claim~1: $2$ divides $a_2$.
\\
- {\it Induction step (of the basic case).}  Assuming that $2$ divides $a_i$ for all $2\leq i\leq k$
we prove that $2$ divides $a_{k+1}$. Two cases can occur.
\\
{\it Subcase 1: $k+1$ is odd.}
Then, $k$ is even, 2 divides $k$ and,  by congruence preservation,  2 divides $f(k+1)-f(1)$.
We have
\\\centerline{$f(k+1)-f(1)
=ka_1+\left(\sum_{i=2}^k a_i {{k+1}\choose i}\right) +a_{k+1}$,} 
 $2$ divides the $a_i$ for $2\leq i\leq k$ by the induction hypothesis, 2 also divides $k$,
hence, $2$ divides $a_{k+1}$.
\\
{\it Subcase 2: $k+1$ is even.}
Then 2 divides $f(k+1)-f(0)$.
Now,
\\\centerline{$f(k+1)-f(0)
=(k+1)a_1+\left(\sum_{i=2}^k a_i {{k+1}\choose i} \right)+a_{k+1}$,}
$k+1$ is even and $2$ divides the $a_i$ for $2\leq i\leq k$ by the induction hypothesis,
thus, $2$ divides $a_{k+1}$.
\\
\textbullet\;{\it Induction step: $p\geq2$ and $p+1<n$. Assume that

\begin{equation}\label{eq:induction step on p}
\text { for all } q\leq p\  ,\  q \text{  divides } a_\ell\  \text{  for all }\ \ell\  \text{   such that } q\leq\ell<n
 \end{equation}
 and  prove that $p+1$ divides $a_k$ for all $k$ such that $ p+1\leq k <n$.}
Again, we use induction on $k\geq p+1$ and we assume that $k$ divides $m$ in order to use congruence preservation. When $k$ does not divide $m$ we factorize $k=a\times b$ with  with $b$ dividing $m$ and $a$ coprime with $m$ and a similar proof will show that $b$ divides $a_k$ and $k$ divides $a_k$ (cf. the proof of Induction in Claim 1).
\\
- {\it Basic case (of the induction step) $k=p+1$.} 
Follows from Claim~1:: $p+1$ divides $a_{p+1}$.
\\
- {\it Induction step (of the induction step).} Assuming that $p+1$ divides $a_i$ for all $i$ such that $p+1\leq i\leq k$,
we prove that $p+1$ divides $a_{k+1}$. As $p+1$ divides $k+1-(k-p)$ and
 $f$ is congruence preserving, $p+1$ divides $f(k+1)-f(k-p)$
which is given by
\begin{multline}\label{eq:p+1 k+1 - k-p}
f(k+1)-f(k-p)=
\sum\limits_{i=1}^{k-p} a_i \left({{k+1}\choose i}- {{k-p}\choose i}\right)\\
\;+\
\left(\sum\limits_{i=k+1-p}^{k} a_i {{k+1}\choose i}\right)
\;+\;a_{k+1}
\end{multline}
\\
First look at the terms of the first sum corresponding to $1\leq i\leq p$.
The induction hypothesis \eqref{eq:induction step on p} on $p$ insures that $q$ divides $a_k$
for all $q\leq p$ and $k\geq q$.
In particular, letting $k=i$ and using  Lemma~\ref{l:lcm modular}, we see that $\lcm(i)$ divides $a_i$ in $\Z/m\Z$.
Since $(k+1)-(k-p)=p+1$, Lemma~\ref{l:n divise Ankb} insures that
$p+1$ divides $\lcm(i)\left({{k+1}\choose i}- {{k-p}\choose i}\right)$.
A fortiori, 
$p+1$ divides $a_i \left({{k+1}\choose i}- {{k-p}\choose i}\right)$.
\\
Now turn to the terms of the first sum corresponding to
$p+1\leq i\leq k-p$ (if there are any).
The induction hypothesis (on $k$) insures that
$p+1$ divides $a_i$ for all $p+1\leq i\leq k$.
Thus, each term of the first sum is divisible by $p+1$.
\\
Consider now the terms of the second sum.
By the induction hypothesis (on $k$),
$p+1$ divides $a_i$ for all $p+1\leq i\leq k$.
%A fortiori, $p+1$ divides the terms associated with the $i$'s
%such that $\max(p+1,k+1-p)\leq i\leq k$.
It remains to look at the terms associated with the $i$'s
such that $k+1-p\leq i\leq p$
(there are such $i$'s in case $k+1-p<p+1$).
For such $i$'s we have $0\leq (k+1)-i\leq (k+1)-p<p+1\leq k+1$
and Lemma \ref{lemme.p}
(used with $k+1,i$ and $p+1$  in place of $n,k$ and $p$)
insures that $p+1$ divides $\lcm(i) {k+1\choose i}$.
Now, for such $i$'s, the induction hypothesis  \eqref{eq:induction step on p} on $p$
insures that $\lcm(i)$ divides $a_i$.
Thus, $p+1$ divides each $a_i {k+1\choose i}$.

Since $p+1$ divides the $k$ first  terms of the right-hand side of
\eqref{eq:p+1 k+1 - k-p} and also divides the left-hand side,
it must divide the last term  $a_{k+1}$ of the right-hand side.
This finishes the proof of the inductive step of the inductive step
hence also the proof of Claim 2, and  of the ``only if" part of the Theorem.
\end{proof}
\smallskip\noindent
%----------------------------------------------------------------
{\bf  ``If" part of Theorem \ref{thm:main}}
%----------------------------------------------------------------
%
Assume all the $a_k$s in equation \eqref{eq:main} are divisible by $lcm(k) $ and prove that $f$ is congruence preserving , i.e. that, for all $a,b\in\{0,\ldots,n-1\}$,  if $a-b$ divides $m$ then $a-b$ divides $f(a) -f(b)$ in $\Z/m\Z$. 

If all the $a_k$s in equation \eqref{eq:main} are divisible by $lcm(k) $ then $f$  can be written in the form
$f(n)= \sum_{k= 0}^n b_k \lcm(k){n \choose k}$.
Consequently,
$$
f(a)-f(b)= \left(\sum_{k= 0}^b 
                         b_k \lcm(k)\Big({a \choose k} -{b\choose k}\Big)\right)
+ \sum_{k= b+1}^a b_k \lcm(k){a \choose k}
$$

By Lemma \ref{lemme.a-b}, $a-b$  divides each term of the first sum.

Consider the terms of the second sum.
For $b+1\leq k\leq a$, we have $0\leq a-k<a-b\leq a$
and Lemma \ref{lemme.p}
(used with $a,k$ and $a-b$  in place of $n,k$ and $p$)
insures that $a-b$ divides $\lcm(k){a \choose k}$.
Hence,  $a-b$ divides each term of the second sum.
\end{proof}
%

%---------------------------------------
\subsection{On a family of generators}\label{ssg:main}
%---------------------------------------
%
We now sharpen the degree of the rat-polynomial representing a congruence preserving  function $\Z/n\Z \to \Z/m\Z$.
We need first some properties of the $lcm$ function and a definition.
\begin{lemma} 
 In $\Z/m\Z$  we have $lcm(k)= u\times \prod p_i^{\alpha_i^k} $ with $u$ invertible in $\Z/m\Z$, $p_i $  prime and dividing $m$, and $ \alpha_i^k=\max\{\beta_i | p_i^{\beta_i}\leq k\}\;$. 
%
%In  $ \Z/m\Z$ we have $$lcm(k)= u\times \prod_{p_i | m, p_i \text{ prime},\ \alpha_i^k=\max\{\beta_i | p_i^{\beta_i}\leq k\}}\; p_i^{\alpha_i^k} $$ with $u$ invertible in $\Z/m\Z$.
%where $p_i$ is a prime dividing $n$ and $p_i^{\alpha_i^k}$ the greatest power of $p_i\leq k$.
\end{lemma}

\begin{definition}\label{def:mu}
For $m\geq1$, with prime factorization  $m=p_1^{\alpha_1}\cdots p_\ell^{\alpha_\ell}$, let
$\mu(m)= \max_{i=1,\ldots,\ell}\;{p_i^{\alpha_i} }$ be the largest power of prime dividing $m$. 
\end{definition}
\begin{example}  In $\Z/m\Z$ the element $\lcm(k)$ is null for $k$ large enough.\\
In  $ \Z/8\Z$\quad 
\begin{tabular}{| c | c | c | c | c | c |c | c | c |}
 \hline			
 $ k$ & 1 & 2 & 3&4&5& 6 & 7  & 8 \\ \hline
$\lcm(k)\cong$ & 1 & 2 & 2&4&4&4&4&0 \\ \hline
\end{tabular}\quad $\mu(8)=8$
\\
 In  $ \Z/12\Z$
\begin{tabular}{| c | c | c | c | c | c |c | c | c | c | c | c |}
 \hline			
 $ k$ & 1 & 2 & 3&4&5& 6 & 7 & 8&9&10 &11 \\ \hline
$ \lcm(k)\cong$ & 1 & 2 & 6&0&0&0&0&0&0&0&0 \\ \hline
\end{tabular}\quad $\mu(12)=4$

\end{example} 

\begin{lemma}\label{l:lcm}
In $\Z/m\Z$, $\mu(m)$ is the least integer such that $\lcm(k)$ is null and $\forall k\geq\mu(m)$, $\lcm(k)$ is null.
\end{lemma}

 \begin{remark} 
 (1) Either $\mu(m)=m$ or $\mu(m)\leq m/2$\,.
\\
(2) In general, $\mu(m)$ is greater than $\mu'(m)$,
the least $k$ such that $m$ divides $k!$ considered in  \cite{ch1995}:
for $m=8$, $\mu(m)=8$ whilst $\mu'(m)=4$. 
\end{remark} 

Using Lemma~\ref{l:lcm}, we can get a better version of
Theorem~\ref{thm:main}.

\begin{theorem}\label{carac} 
Function $f\colon \Z/n\Z \to \Z/m\Z$ is congruence preserving  iff it can be represented by a  rational  polynomial $P=\sum_{k=0}^{p} a_k \binom{x}{k}$  with degree  $p<\min(n,\mu(m))$ and such that $lcm(k)$ divides $a_k$ for all $k\leq p$. 
\end{theorem}

\begin{proof}
For $k\geq \mu(m) $, $m$ divides $lcm(k)$ hence the coefficient $a_k$ is 0.
\end{proof}
\begin{theorem} (1) Every congruence preserving function $f:\Z/n\Z\to\Z/m\Z$
is rat-polynomial with degree less than $\mu(m)$.
\\
(2) The family of rat-polynomial functions
$\+F=\{lcm(k)(P_k) |  0\leq k<\min(n,\mu(m))\}$ generates the set of congruence preserving functions.
\\
(3) $\+F$ is a basis of  the set of congruence preserving functions
if and only if
$m$ has no prime divisor $p<\min(n,m)$
(in case $n\geq m$ this means that $m$ is prime)\,.

\end{theorem}

\begin{proof}
(1) and (2) are restatements of Theorem \ref{carac} . We prove (3). 
\\
``Only If" part. Asssuming $m$ has a prime divisor $p<\min(n,m)$,  let $p$ be the least one.
Then $\+F$ is not linearly independant.  In $\Z/m\Z$, $\lcm({p})\neq0$ hence $\lcm({p})\,P_p$\,
is not the null function since $P_p({p})=1$. However $(m/p)\,\lcm({p})=0$ hence
$(m/p)\,\lcm({p})\,P_p$ is the null function. As $(m/p)\neq0$, we see that $\+F$ cannot be a basis.
\smallskip\\
``If" part.
Assume that $m$ has no prime divisor $p<\min(n,m)$\,.
We prove that $\+F$ is $\Z/m\Z$-linearly independent.
Suppose that the $\Z/m\Z$-linear combination
$
L\ =\ \sum_{k=0}^{\min(n,\mu(m))-1}
a_k\,\lcm(k)\,P_k
$
is the null function $\Z/n\Z\to\Z/m\Z$\,.
By induction on $k=0,\ldots,\min(n,\mu(m))-1$
we prove that $a_k=0$\,.
\\
\textbullet\ {\it Basic cases $k=0,1$.}
Since $L(0)=a_0$ 
we get $a_0=0$\,.
Since $L(1)=a_0+a_1\,1$ 
we get $a_1=0$\,.
\\
\textbullet\ {\it Induction step.}
Assuming that $k\geq2$ and $a_i=0$ for $i=0,\ldots,k-1$,
we prove that $a_k=0$\,.
Note that  $P_\ell(k)=\binom{k}{\ell}$ for $k<\ell<n$.
Since $a_i=0$ for $i=0,\ldots,k-1$,
and $P_k(k)=1$
we get $L(k)=a_k\,\lcm(k\,)$\,.
Since $k<\min(n,\mu(m))$
and $m$ has no prime divisor $p<\min(n,m)$,
the numbers $\lcm(k)$ and $m$ are coprime
hence $\lcm(k)$ is invertible in $\Z/m\Z$ and equality $L(k)=a_k\lcm(k)=0$ implies $a_k=0$\,.
\end{proof}

%%%%%%%%%%%%%
%%%%%%%%%%%%%
\section{Counting congruence preserving functions}\label{s:count}
%%%%%%%%%%%%%
%%%%%%%%%%%%%

We are now interested in the number of congruence preserving   functions  $\Z/n\Z\to\Z/m\Z$.
As two different  rational polynomials correspond to different functions by Theorem \ref{uniqueRepresentation} (unicity of the representation by a rational polynomial),
 the number of congruence preserving  functions  $\Z/n\Z\to\Z/m\Z$   is equal to the number of polynomials representing them. 
\begin{proposition}\label{ZnToZmCount}
Let $CP(n,m)$ be the number of  congruence preserving  functions  $\Z/n\Z\to\Z/m\Z$\,.
For $m=p_1^{e_1}p_2^{e_2}\cdots p_\ell^{e_\ell}$, we have
\begin{eqnarray*}
\!\!\!CP(n,m)&=&{p_1^{p_1+p_1^2+\cdots+p_1^{e_1}} \times\cdots\times p_\ell^{p_\ell+p_\ell^2+\cdots+p_\ell^{e_\ell}} }\ \text{if } n\geq \mu(m)\\
\!\!\!CP(n,m)&=&\!\!\!\prod\limits_{\{i\mid  p_i^{e_i}<\mu(m)\}} {p_i^{p_i+p_i^2+\cdots+p_i^{e_i}}} \times \!\!\!\prod\limits_{\{i\mid  p_i^{e_i}\geq\mu(m)\}}{p_i^{p_i+p_i^2+\cdots+p_i^{\lfloor \log_{p} n\rfloor}+n(e-\lfloor \log_{p} n\rfloor)}}\notag\\
&&\text{ if }  \ n< \mu(m) 
\end{eqnarray*}
Equivalently, using an \`a la Vinogradov's notation for better readability and  writing $E(p,\alpha)$ in place of $p^\alpha$ we have
\begin{eqnarray*}
\!\!\!CP(n,m)&=&\prod\limits_{i=1}^\ell E(p_i,\sum\limits_{k=1}^{e_i}p_i^k )\quad \text { if } n\geq \mu(m)\\
\!\!\!CP(n,m)&=&\!\!\!\prod\limits_{\{i\mid  p_i^{e_i}<\mu(m)\}} E(p_i,\sum\limits_{k=1}^{e_i}p_i^k ) \times \!\!\!\prod\limits_{\{i\mid  p_i^{e_i}\geq\mu(m)\}} E(p_i,\sum\limits_{k=1}^{\lfloor \log_{p} n\rfloor}  p_i^k+n(e-\lfloor \log_{p} n\rfloor))\notag\\&&
\text{ if }  \ n< \mu(m) 
\end{eqnarray*}

\end{proposition}
\begin{corollary}
For $ n\geq \mu(m)$, $CP(n,m)$ does {\em not } depend  on $n$.
\end{corollary}
\begin{proof}[Proof of Proposition~\ref{ZnToZmCount}]
By Theorem \ref{carac}, we must count the number of $n$-tuples of coefficients $(a_0,\ldots,a_{n-1})$,  with $a_k$ a multiple of $lcm(k)$  in $\Z/m\Z$. 

\medskip
\noindent
{\it{\normalfont\bf Claim 1.} 
 For $m=p_1^{e_1}p_2^{e_2}\cdots p_\ell^{e_\ell}$, for all $n$,
 $CP(n,m)=\Pi_{i=1}^{i=\ell}\ CP(n,p_i^{e_i})$ }\,.
\begin{proof}
Let  $\lambda(m,k)$ be the number of multiples of $lcm(k)$ in $\Z/m\Z$,
i.e. order of the subgroup generated by $lcm(k)$ in $\Z/m\Z$\,. 

Since $\Z/m\Z$ is isomorphic to $\Pi_{i=1}^{i=\ell}\Z/p_i^{e_i}\Z$,
we have $\lambda(m,k)=\Pi_{i=1}^{i=\ell}\lambda({p_i^{e_i}},k)$
for each $k$.
Thus, the number of $n$-tuples  $(a_0,\ldots,a_{n-1})$
such that $lcm(k)$ divides $a_k$ is equal to
$$\Pi_{k<n} \lambda(m,k)=\Pi_{k<n}\ \Pi_{i=1}^{i=\ell}\lambda({p_i^{e_i}},k)=\Pi_{i=1}^{i=\ell}\ \Pi_{k<n}\lambda({p_i^{e_i}},k)$$
The  trick in the proof is the permutation of the two  products; hence the Claim by using Theorem \ref{thm:main}.
%Note that, for all $j$ , $\lambda_j(0)=\lambda_j(1)=j$ and $\lambda_j(k)=1$ if $j$ divides $lcm(k)$ (this holds in particular if $k\geq j$). Hence $\Pi_{k<p^\alpha}\lambda_{p_i^{e_i}}(k)=\Pi_{k<p_i^{e_i}}\lambda_{p_i^{e_i}}(k)$ is the product of the orders of the subgroups generated by $lcm(k)$ in $\Z/p_i^{e_i}\Z$. 
\end{proof}

Claim 1 reduces the problem to counting the congruence preserving  functions $\Z/ n\Z\to \Z/{p_i^{e_i}}\Z$.  We will now use Proposition \ref{carac} for this counting.

\medskip
\noindent
{\it{\normalfont\bf Claim 2.} $$CP(n,p^e)=\begin{cases}p^{p+p^2+\cdots+p^e}&{\text{ if }}n\geq p^e\\
p^{p+p^2+\cdots+p^l+(e-l)n}&{\text{ if }} p^l\leq n< p^e {\text{with }} l=\lfloor \log_{p} n\rfloor.
\end{cases}$$}

\vspace {-0.5cm}
\begin{proof} By Theorem  \ref{carac}, as $\mu(p^e)=p^e$, letting $\nu=\inf(n,p^e)$, $CP(n,p^e)=CP(\nu,p^e)=\Pi_{k<\nu}\lambda(p^{e},k)$. 
For $p^j\leq k< p^{j+1}$ the order $\lambda({p^e},k)$ of the subgroup generated by $lcm(k)$ in $\Z/p^{e}\Z$ is $p^{e-j}$ and  there are $p^{j+1}-p^j$ such $k$'s.
\\
$\bullet$  Assume first $n\geq p^e$, then $CP(n,p^e)=CP(p^e,p^e)=p^M$ with 

\begin{eqnarray*}M&=&{e}p+(e-1)(p^2-p)+\cdots+(e-j)(p^{j+1}-p^{j})+\cdots+p^{e}-p^{e-1}\\
&=&ep+\sum_{j=1}^{e-1}(e-j)(p^{j+1}-p^{j})=p+p^2+\cdots+p^e
\end{eqnarray*}
$\bullet$  Assume then $p^l \leq n< p^e$, with $l=\lfloor \log_{p} n\rfloor$; 
then $CP(n,p^e)=p^M$ with 
\begin{eqnarray*}M&=&ep+\sum_{j=1}^{l-1}(e-j)(p^{j+1}-p^{j})+(e-l)(n-p^l)\\
&=&p+p^2+\cdots+p^l+n(e-l)\,.
\end{eqnarray*}

\vspace {-0,5cm}\end{proof}
This finishes the proof of Proposition~\ref{ZnToZmCount}.
\end{proof}

\begin{remark}
In \cite{Bh1997} the number of congruence preserving functions $ \Z/n\Z\to \Z/p^e\Z$ is shown to be equal to
$p^{en-\sum_{k=1}^{n-1}\min\{e,\lfloor \log_pk\rfloor\}}$.  Note that
 for $p^i\leq k<p^{i+1}$, $\lfloor \log_pk\rfloor=i$, hence:
  for $k\leq p^e$, $\min\{e,\lfloor \log_pk\rfloor\} = \lfloor \log_pk\rfloor$ and 
  for $k\geq p^e$, $\min\{e,\lfloor \log_pk\rfloor\} = e$.
  We thus have \\
$\bullet$ if $n\geq p^e$, 
\begin{eqnarray*}\textstyle\sum_{k=1}^{n-1}\min\{e,\lfloor \log_pk\rfloor\}&=&\textstyle\sum_{k=1}^{p^e-1}\lfloor \log_pk\rfloor+\textstyle\sum_{k={p^e}}^{n-1}e\\
&=&\textstyle\sum_{j=0}^{e-1} j\times (p^{j+1}-p^j) + e\times(n-p^e)\\
&=&
 0+(p^2-p)+2 (p^3-p^2)+\cdots+\\
 && \quad+(e-1)(p^e-p^{e-1})+ e(n-p^e)
 \end{eqnarray*}
\begin{eqnarray*}
\text{hence }\qquad\qquad
 {en-\textstyle\sum_{k=1}^{n-1}\min\{e,\lfloor \log_pk\rfloor\}}&=&  p+\cdots+p^e \\
\text {and }\qquad\qquad\qquad
 p^{en-\sum_{k=1}^{n-1}\min\{e,\lfloor \log_pk\rfloor}
&=&p^{p+p^2+\cdots+p^e}
 \end{eqnarray*}
which coincides with our counting in Claim 2.\\
$\bullet$ if $n< p^e$, and $l=\lfloor \log_pn\rfloor$, then  
\begin{eqnarray*}\textstyle\sum_{k=1}^{n-1} \lfloor\log_pk\rfloor&=&\textstyle\sum_{k=1}^{l-1}\lfloor \log_pk\rfloor+\textstyle\sum_{k=l}^{n-1} \lfloor \log_pk\rfloor\\
&=&\textstyle\sum_{j=0}^{l-1} j\times (p^{j+1}-p^j) + l\times(n-p^l) \\
&=&
 0+(p^2-p)+2 (p^3-p^2)+\cdots+(l-1)(p^l-p^{l-1})\\
 && \quad+ l(n-p^e)\\
 &=&-(p+\cdots+p^l)+nl \end{eqnarray*}
 and $en-\sum_{k=1}^{n-1} \lfloor\log_pk\rfloor=p+\cdots+p^l+(e-l)n$, which again coincides with our counting in Claim 2.
\end{remark}

%%%%%%%
\section{Conclusion}
We proved that the rational polynomials $\lcm(k)\,P_k$ generate
the $(\Z/m\Z)$-submodule of congruence preserving functions
$\Z/n\Z\to\Z/m\Z$.
When $n$ is larger than the largest prime power dividing $m$,
the number of functions in this submodule is independent of $n$.
An open problem is the existence of a basis of this submodule.

%%%%%%%%%%%%%%%%%%%%%%%%%%%%%%u%%%
%%%%%%%%%%%%%%%%%%%%%%%%%%%%%%%%%%
%%%%%%%%%%%%%%%%%%%%%%%%%%%%%%%%%%
%%%%%%%%%%%%%%%%%%%%%%%%%%%%%%%%%%
%%%%%%%%%%%%%%%%%%%%%%%%%%%%%%%%%%
%%%%%%%%%%%%%%%%%%%%%%%%%%%%%%%%%%
%%%%%%%%%%%%%%%%%%%%%%%%%%%%%%%%%%
%%%%%%%%%%%%%%%%%%%%%%%%%%%%%%%%%%

%{\scriptsize\tableofcontents}
\end{document}